\documentclass[11pt]{article}
\usepackage{geometry}                
\usepackage{graphicx}
\usepackage{amssymb}
\usepackage{amsmath}
\usepackage{amsthm}
\usepackage{epstopdf}
\newtheorem{theorem}{Theorem}[section]
\newtheorem{lemma}[theorem]{Lemma}
\newtheorem{example}[theorem]{Example}

\title{Note on {\it Interacting Langevin Diffusions: Gradient Structure and Ensemble Kalman Sampler} by 
Garbuno-Inigo, Hoffmann, Li and Stuart}
\author{Nikolas N\"usken\thanks{Institute of Mathematics, University of Potsdam, Karl-Liebknecht-Str. 24/25, D-14476 Potsdam, Germany, 
{\tt nuesken@uni-potsdam.de}} \and Sebastian Reich\thanks{
Institute of Mathematics, University of Potsdam, Karl-Liebknecht-Str. 24/25, D-14476 Potsdam, Germany, {\tt sereich@uni-potsdam.de}
}
}

\begin{document}
\maketitle

\begin{abstract}
An interacting system of Langevin dynamics driven particles has been proposed for sampling from a given posterior density
by Garbuno-Inigo, Hoffmann, Li and Stuart in {\it Interacting Langevin Diffusions: Gradient Structure and Ensemble Kalman Sampler}
(arXiv:1903:08866v2). The proposed formulation is primarily studied from a formal mean-field limit perspective, while the theoretical 
behaviour under a finite particle size is left as an open problem.  In this note we demonstrate that the particle-based covariance interaction term 
requires a non-trivial correction. We also show that the corrected dynamics samples exactly from the
desired posterior provided that the empirical covariance matrix of the particle system remains non-singular and the posterior log-density 
satisfies the standard Bakry--{\'E}mery criterion.
\end{abstract}


\section{Introduction}
In \cite{sr:GIHLS19}, the authors propose to evolve an interacting set of particles $U = \{u^{(j)}\}_{j=1}^J$ according to the following system
of stochastic differential equations (SDEs):
\begin{equation} \label{eq:IPS1}
\dot{u}^{(j)} = -C(U)\nabla \Psi_R(u^{(j)}) + \sqrt{2C(U)} \dot{\bf W}^{(j)},
\end{equation}
where $\Psi_R$ is a suitable potential and the ${\bf W}^{(j)}$ are a collection of i.i.d.~standard Brownian motions in the state space
$\mathbb{R}^d$, aiming at approximating the posterior $\pi^\ast \propto \exp(-\Psi_R)$ in the long-time limit. The matrix $C(U)$ is chosen to be the empirical covariance between particles,
\begin{equation}
C(U) = \frac{1}{J} \sum_{k=1}^J (u^{(k)}-\bar{u})(u^{(k)}-\bar u)^{\rm T} \in \mathbb{R}^{d\times d},
\end{equation}
where $\bar u$ denotes the sample mean
\begin{equation}
\bar u = \frac{1}{J}\sum_{j=1}^J u^{(j)}.
\end{equation}
We require $J>d$ in order for $C(U)$ to have full rank generically. But we also note that (\ref{eq:IPS1}) is valid
even for $J\le d$ in which case the dynamics is restricted to a subspace of $\mathbb{R}^d$.
 
Formally taking the large particle limit $J\to \infty$ leads to the mean-field equation
\begin{equation}
\dot{u} = -\mathcal{C}(\rho)\nabla\Psi_R(u) + \sqrt{2\mathcal{C}(\rho)}\dot{W},
\end{equation}
with corresponding nonlinear Fokker--Planck equation 
\begin{equation} \label{eq:FPE1}
\partial_t \rho = \nabla \cdot \left( \rho\, \mathcal{C} (\rho)\nabla \Psi_R  + \rho \,\mathcal{C}(\rho)\nabla \ln \rho \right)
= \nabla \cdot \left( \rho \,\mathcal{C}(\rho) \nabla \frac{\delta E}{\delta \rho}\right)
\end{equation}
in the marginal densities $\rho$ of $u$ at time $t\ge 0$.
Here we have defined the macroscopic mean and covariance
\begin{equation}
m(\rho) = \int_{\mathbb{R}^d} u \rho(u) \,{\rm d}u, \qquad \mathcal{C}(\rho) = 
\int_{\mathbb{R}^d} (u-m(\rho))(u-m(\rho))^{\rm T} \rho(u)\, {\rm d}u
\end{equation}
and the energy functional
\begin{equation}
E(\rho) = \int_{\mathbb{R}^d} \rho(u) \left(\Psi_R(u) + \ln\rho(u)\right) {\rm d}u.
\end{equation}
The Fokker--Planck equation (\ref{eq:FPE1}) and its gradient flow properties are carefully studied in \cite{sr:GIHLS19} and we refer in particular to
Propositions 2, 4 and 7.

\section{Properties of the finite-size particle system}

In this note we wish to demonstrate that similar properties hold already at the level of the finite-size interacting particle system
(\ref{eq:IPS1}) provided one introduces an appropriate correction term. Our first step is to check whether the product measure
\begin{equation} \label{eq:TPDF}
\pi(U) = \prod_{j=1}^J \pi^\ast(u^{(j)})
\end{equation}
is invariant under (\ref{eq:IPS1}). Recall that $\pi^\ast(u)$ denotes 
the canonical measure associated with the potential $\Psi_R$, that is, $\pi^\ast(u) \propto \exp(-\Psi_R(u))$.
We now rewrite (\ref{eq:IPS1}) in the form
\begin{equation} \label{eq:IPS2}
\dot{U} = S(U)\nabla \ln \pi(U) + \sqrt{2S(U)}\dot{\bf W},
\end{equation}
where $S(U) \in \mathbb{R}^{D\times D}$, $D = dJ$, is a block diagonal matrix with $J$ $d\times d$ entries $C(U)$ and ${\bf W}$ denotes
standard $D$-dimensional Brownian motion. The associated Fokker--Planck equation \cite{sr:P14} 
for the time evolution of the joint density $\mu$ of $U$ can be written as
\begin{equation} \label{eq:FPE0}
\partial_t \mu = \nabla \cdot \left( \mu S \nabla \left\{\ln \mu - \ln \pi \right\} + \mu \nabla \cdot S \right),
\end{equation}
where the vector-valued divergence of the matrix $S$ is defined by
\begin{equation}
(\nabla \cdot S)_i = \sum_{j=1}^D \partial_j S_{ij}, \quad i= 1,\ldots, D.
\end{equation}
From \eqref{eq:FPE0} we see immediately that $\nabla \cdot (\pi \nabla \cdot S) = 0$ is a necessary condition for the invariance of $\pi$.
The following lemma addresses this issue for the specific form of the SDE (\ref{eq:IPS1}).

\begin{lemma} \label{lemma1}
It holds that
\begin{equation}
\nabla \cdot S\,(U) = \frac{d+1}{J}\left(U-\bar U\right) \in \mathbb{R}^D,
\end{equation} 
where $\bar U$ denotes the $D$ dimensional vector consisting of $J$ copies of $\bar u \in \mathbb{R}^d$. In particular, since $\nabla \cdot S \neq 0$, $\pi$ is in general not invariant for the dynamics defined by \eqref{eq:IPS1}.
\end{lemma}

\begin{proof} We note the properties 
\begin{align}
\nabla_{u^{(j)}} \cdot \left( u^{(j)} (u^{(j)})^{\rm T}\right) &= (d+1) u^{(j)},\\
\nabla_{u^{(j)}} \cdot \left( u^{(j)} (u^{(k)})^{\rm T}\right) &= u^{(k)},\\
\nabla_{u^{(j)}} \cdot \left( u^{(k)} (u^{(j)})^{\rm T}\right) &= d u^{(k)},
\end{align}
$k\not=j$, for the divergence operator. Using
\begin{equation}
C(U) = \frac{1}{J} \sum_{k=1}^M u^{(k)} (u^{(k)})^{\rm T} - \bar u \bar u^{\rm T},
\end{equation}
an explicit calculation reveals then that
\begin{align}
\nabla_{u^{(j)}} \cdot C\,(U) &= 
\frac{d+1}{J} (u^{(j)} - \bar u)  \in \mathbb{R}^d
\end{align}
and hence 
\begin{equation}
\nabla \cdot S \,(U) = \frac{d+1}{J}\left(U-\bar U\right) \in \mathbb{R}^D.
\end{equation}

\end{proof}

\begin{example} Let us consider a one-dimensional Gaussian target measure ($d=1$) with variance $b^2$, i.e. the potential is given by $\Psi_R(u) =  u^2/(2b^2)$. Under the approximating assumptions that the equilibrium density associated to \eqref{eq:IPS1} is a product of $J$ Gaussians with variance $\sigma^2$, and that the covariance at equilibrium can be well-approximated by the state-independent stationary covariance (mean-field assumption),  one can derive the  relation 
\begin{equation}
\sigma^2 \approx \frac{J-2}{J} b^2,
\end{equation}	
which we numerically verified for $J=4$. Note that $\sigma^2 < b^2$, i.e.
the dynamics \eqref{eq:IPS1} systematically underestimates the posterior variance. Since $ \sigma^2 \xrightarrow{J \rightarrow \infty} b^2$, this error vanishes in the large-particle regime, in agreement with the results in \cite{sr:GIHLS19}.
\end{example}

The invariance of $\pi$ can be restored for the interacting Langevin diffusion model (\ref{eq:IPS2}) 
if one replaces the drift term $S \nabla \ln \pi$ in (\ref{eq:IPS2}) by $S \nabla \ln \pi + \nabla \cdot S$. More specifically, 
Lemma \ref{lemma1} implies that the corrected dynamics
\begin{equation} \label{eq:IPS1c}
\dot{u}^{(j)} = -C(U)\nabla \Psi_R(u^{(j)}) + \frac{d+1}{J} (u^{(j)}-\bar u) + \sqrt{2C(U)} \dot{\bf W}^{(j)},
\end{equation}
or, written more compactly
\begin{equation} \label{eq:IPS2c}
\dot{U} = S(U)\nabla \ln \pi(U) + \frac{d+1}{J}(U-\bar U) + \sqrt{2S(U)}\dot{\bf W},
\end{equation}
not only samples from the correct target density (\ref{eq:TPDF}) but also 
gives rise to a gradient flow structure for the time evolution of the joint density $\mu$ over all particles at time $t \ge 0$.

\begin{lemma} \label{lemma2}
The Fokker--Planck equation associated to the time evolution of the joint density $\mu$ under the corrected interacting particle 
formulation (\ref{eq:IPS2c})  is of the form
\begin{equation} \label{eq:FPE2}
\partial_t \mu = \nabla \cdot \left( \mu\, S \nabla \left\{ \ln \mu  - \ln \pi \right\} \right)
= \nabla \cdot \left( \mu \,S \nabla \frac{\delta V}{\delta \mu}\right),
\end{equation}
where the energy functional $V$ is the Kullback--Leibler divergence between $\mu$ and $\pi$, given by
\begin{equation}
V(\mu) = \int_{\mathbb{R}^D}  \mu(U) \{\ln \mu (U) - \ln \pi(U) \} \, {\rm d}U.
\end{equation}
\end{lemma}

\begin{proof} Follows from the Fokker--Planck equation (\ref{eq:FPE0}) for the SDE system (\ref{eq:IPS2}), the substitution of
the drift term $S \nabla \ln \pi$ by $S \nabla \ln \pi + \nabla \cdot S$, and Lemma \ref{lemma1}. 
See, for example, \cite{sr:P14} for a derivation of the Fokker--Planck equation for SDEs with multiplicative noise, which leads to
(\ref{eq:FPE0}).
\end{proof}

Hence the Kalman--Wasserstein gradient flow structure also carries over to the corrected finite-size particle system (\ref{eq:IPS2c}).
However, it cannot be guaranteed that $C(U)$ remains strictly positive definite for all times, in general, even under the condition $J>d$. 
This suggests a regularisation of the form
\begin{equation}
C_{\alpha}(U) = \alpha C_0 + (1-\alpha)C(U),
\end{equation}
where $C_0 \in \mathbb{R}^{d \times d}$ is a fixed symmetric strictly positive definite matrix and $\alpha \in (0,1)$ is a parameter. Under the Bakry--{\'E}mery condition \cite{sr:P14} 
stated in Proposition 2 of \cite{sr:GIHLS19} one can then also establish exponential convergence to the equilibrium measure 
for the finite-size particle system (\ref{eq:IPS1c}) with $C(U)$ replaced by $C_\alpha(U)$ and the correction term 
by $(1-\alpha) (d+1) (u^{(j)}-\bar u)/J$. This follows from standard arguments for the classical Fokker--Planck equation
\cite{sr:P14}. See also Proposition 2 of \cite{sr:GIHLS19}. Let us also mention that the convexity requirement on $\Psi_R$ can be relaxed by a perturbation argument due to Holley–-Stroock \cite{sr:P14}. 

The need for a correction term in (\ref{eq:IPS1}) can be avoided by ensuring that the $j$th block entry $C(U)$ in $S(U)$ 
does not depend on the $j$th particle position $u^{(j)}$. This can, for example, be achieved by defining the following modified 
covariance matrices
\begin{equation}
C_{[j]}(U) = \frac{1}{J-1} \sum_{k\not=j} (u^{(k)}-\bar{u}_{[j]})(u^{(k)}-\bar u_{[j]})^{\rm T} \in \mathbb{R}^{d\times d}, \quad j=1,\ldots,J, 
\end{equation}
where $\bar u_{[j]}$ denotes the leave-one-out sample mean
\begin{equation}
\bar u_{[j]} = \frac{1}{J-1}\sum_{k\not=j}^J u^{(k)},
\end{equation}
and by replacing (\ref{eq:IPS1}) with
\begin{equation} \label{eq:IPS3}
\dot{u}^{(j)} = -C_{[j]}(U)\nabla \Psi_R(u^{(j)}) + \sqrt{2C_{[j]}(U)} \dot{\bf W}^{(j)}.
\end{equation}
We note that very similar ideas have been employed, for example, in \cite{sr:LMW18} for second-order Langevin dynamics. However, while
the reformulation (\ref{eq:IPS3}) is appealing, its computational implementation is more demanding due to the need for computing $J$ 
different covariance matrices and their square roots.  On the other hand, if the destabilising
effect of the correction term in (\ref{eq:IPS1c}) leads to numerical difficulties under small or moderate particle sizes, 
then (\ref{eq:IPS3}) could be taken as a starting point for formulating gradient-free formulations in the spirit of the ensemble Kalman 
sampler (EKS) as proposed in \cite{sr:GIHLS19}.

\medskip \medskip

\noindent
{\bf Acknowledgement.} This research has been partially funded by 
Deutsche Forschungsgemeinschaft (DFG) through grants 
CRC 1294 \lq Data Assimilation\rq (project B04) and CRC 1114 \lq Scaling Cascades\rq (project A02) .

%
%
\bibliographystyle{plain}
\bibliography{bib_paper}
%
%

\end{document}